\documentclass[12pt]{article} 

\usepackage[margin=1in]{geometry}

\usepackage{setspace}
\usepackage{amssymb,amsfonts,amsmath,latexsym,amsthm}
\usepackage[]{graphicx}
\usepackage{bm}
\usepackage{dsfont}
\usepackage[round,comma]{natbib}
\usepackage[utf8]{inputenc}
\usepackage[colorlinks,citecolor=blue,linkcolor=blue,urlcolor=blue]{hyperref}

\theoremstyle{plain}
\newtheorem{theorem}{Theorem}[section]

\newtheorem{lemma}[theorem]{Lemma}
\newtheorem{proposition}[theorem]{Proposition}

\theoremstyle{definition}

\theoremstyle{remark}

\allowdisplaybreaks 

\DeclareMathOperator*{\PY}{PY}
\renewcommand{\v}{\bm v}
\newcommand{\z}{\bm z}
\newcommand{\y}{\bm y}
\newcommand{\x}{\bm x}
\newcommand{\C}{\bm C}

\DeclareMathOperator*{\Beta}{Beta}

\newcommand{\N}{\mathbb{N}}

\newcommand{\E}{\mathbb{E}}
\renewcommand{\Pr}{\mathbb{P}}
\newcommand{\I}{\mathds{1}}

\newcommand{\ddr}{\mathrm{d}}

\newcommand{\iid}{\stackrel{\mathrm{iid}}{\sim}}
\newcommand{\ind}{\stackrel{\mathrm{ind}}{\sim}}

\title{A simple proof of Pitman--Yor's Chinese restaurant process from its stick-breaking representation}
\author{Caroline Lawless and Julyan Arbel\\
Univ. Grenoble Alpes, Inria, CNRS, LJK, 38000 Grenoble, France}
\date{}

\begin{document}
\maketitle
\begin{abstract}
For a long time, the Dirichlet process has been the gold standard discrete random measure in Bayesian nonparametrics. The Pitman--Yor process provides a simple and mathematically tractable generalization, allowing for a very flexible control of the clustering behaviour. Two commonly used representations of the Pitman--Yor process are the stick-breaking process and the Chinese restaurant process. The former is a constructive representation of the process which turns out very handy for practical implementation, while the latter describes the partition distribution induced. Obtaining one from the other is usually done indirectly with use of measure theory. In contrast, we provide here an elementary proof of  Pitman--Yor's Chinese Restaurant process from its stick-breaking representation.  
\end{abstract}

\section{Introduction}
The Pitman--Yor process defines a rich and flexible class of random probability measures which was developed by \citet{perman1992size} and further investigated by \citet{Pit:95,pitman1997two}. 
It is a simple  generalization of the Dirichlet process \citep{ferguson1973bayesian}, whose mathematical tractability contributed to its popularity in machine learning theory \citep{caron2017generalized}, probabilistic models for linguistic applications \citep{teh2006hierarchical,wood2011sequence}, excursion theory \citep{perman1992size,pitman1997two}, measure-valued diffusions in population genetics \citep{petrov2007two,feng2010some}, combinatorics \citep{vershik2004markov,kerov2006coherent} and statistical physics \citep{derrida1981random}. 

Its most prominent role is perhaps in Bayesian nonparametric statistics where it is used as a prior distribution, following the work of \cite{Ish:Jam:01}. 
Applications in this setting embrace a variety of inferential problems, including species sampling \citep{Fav:etal:09,Nav:etal:08,arbel2015discovery}, survival analysis and graphical models in genetics \citep{Jar:etal:10,ni2018heterogeneous}, image segmentation \citep{Sud:Jor:09}, curve estimation \citep{canale2017pitman}, exchangeable feature allocations \citep{battiston2018characterization} and time-series and econometrics \citep{caron2017generalized,leisen}. 

Last but not least, the Pitman--Yor process is also employed  in the context of nonparametric mixture modeling, thus generalizing the celebrated Dirichlet process mixture model of \cite{Lo(84)}.  Nonparametric mixture models based on the Pitman--Yor process are characterized by a more flexible parameterization than the Dirichlet process mixture model, thus allowing for a better control of the clustering behaviour \citep{deblasi2015gibbs}.  
In addition, see \cite{Ish:Jam:01,favaro2013slice,arbel2019stochastic} for posterior sampling algorithms, \cite{scricciolo2014adaptive,miller2014inconsistency} for asymptotic properties, and \cite{scarpa2009bayesian,canale2017pitman} for spike-and-slab extensions.

The Pitman--Yor process has the following stick-breaking representation:
if
$\v_i \ind \Beta(1-d,\alpha+id)$ for $i = 1,2,\ldots$ with $d\in(0,1)$ and $\alpha>-d$, if $\bm\pi_j = \v_j\prod_{i = 1}^{j-1} (1 -\v_i)$ for $j = 1,2,\ldots$, and if $\bm\theta_1,\bm\theta_2,\ldots\iid H$, then the discrete random probability measure
\begin{align}\label{equation:stick}
\bm P = \sum_{j = 1}^\infty \bm\pi_j \delta_{\bm\theta_j}
\end{align}
is distributed according to the Pitman--Yor process, $\PY(\alpha,d,H)$, with concentration parameter $\alpha$, discount parameter $d$, and base distribution $H$. 

The Pitman--Yor process induces the following partition distribution: 
if $\bm P\sim\PY(\alpha,d,H)$, for some nonatomic probability distribution $H$, we observe data $\x_1,\ldots,\x_n|\bm P\,\iid\, \bm P$, and $\C$ is the partition of the first $n$ integers $\{1,\ldots,n\}$ induced by data, then
\begin{align}\label{eq:CRP}
\Pr(\C = C) = \frac{d^{\mid C \mid}}{(\alpha)_{(n)}}\Big(\frac{\alpha}{d}\Big)_{(\mid C \mid)}\prod_{c \in C}(1-d)_{(\mid c \mid - 1)},
\end{align}
where the multiplicative factor before the product in~\eqref{eq:CRP} is also commonly (and equivalently) written as $(\prod_{i=1}^{\mid C \mid-1}\alpha+id)/(\alpha)_{(n-1)}$  in the literature. When the discount parameter $d$ is set to zero, the Pitman--Yor process reduces to the Dirichlet process and the partition distribution~\eqref{eq:CRP} boils down to the celebrated Chinese Restaurant process \citep[CRP, see][]{antoniak1974mixtures}. By abuse of language, we call  the partition distribution~\eqref{eq:CRP} the Pitman--Yor's CRP. 
Under the latter partition distribution, the number of parts in a partition $C$ of $n$ elements, $k_n=\vert C \vert$, grows to infinity as a power-law of the sample size, $n^d$ \citep[see][for details]{pitman2003poisson}. This Pitman--Yor power-law growth is more in tune with most of empirical data \citep{clauset2009power} than the logarithmic growth induced by the Dirichlet process CRP, $\alpha \log n$.

The purpose of this note is to provide a simple proof of Pitman--Yor's CRP~\eqref{eq:CRP} from its stick-breaking representation~\eqref{equation:stick} (Theorem~\ref{theorem:main}). This generalizes the derivation by \cite{miller2018elementary} who obtained the Dirichlet process CRP \citep{antoniak1974mixtures} from Sethuraman's stick-breaking representation \citep{sethuraman1994constructive}. In doing so, we also provide the marginal distribution of the allocation variables vector~\eqref{eq:allocation-variables} in Proposition~\ref{prop:A}.

\section{Partition distribution from stick-breaking}
% \section{Result}
\label{section:main}

Suppose we make $n$ observations, $z_1,\ldots,z_n.$ We denote the set $\{1,\ldots,n \}$ by $[n]$. Our observations induce a partition of $[n]$, denoted $C= \{c_1,\ldots,c_{k_n}\}$ where $c_1,\ldots,c_{k_n}$ are disjoint sets and $\bigcup_{i=1}^{k_n} c_i = [n],$ in such a way that $z_i$ and $z_j$ belong to the same partition if and only if $z_i=z_j.$ We denote the number of parts in the partition $C$ by ${k_n}=|C|$ and we denote the number of elements in partition $j$ by $|c_j|.$ We use bold font to represent random variables. We write $(x)_{(n)}= \prod_{j=0}^{n-1}(x+j)$ to denote the rising factorial.

\begin{theorem}
    \label{theorem:main}
    Suppose
    \begin{align*}
        & \v_{i} \ind\Beta(1-d,\alpha+id) \text{ for } i=1,2,\ldots, \\ 
        & \bm\pi_j = \v_j\prod_{i = 1}^{j-1} (1 -\v_i) \text{ for } j = 1,2,\ldots
    \end{align*}
    Let allocation variables be defined by
    \begin{align}\label{eq:allocation-variables}
        \z_1,\ldots,\z_n|\bm\pi=\pi \,\iid\, \pi, \text{ meaning, } \Pr(\z_i = j\mid \pi) = \pi_j,
    \end{align}
    and $\C$ denote the random partition of $[n]$ induced by $\z_1,\ldots,\z_n$. Then
\begin{align*}
\Pr(\C = C) = \frac{d^{\mid C \mid}}{(\alpha)_{(n)}}\Big(\frac{\alpha}{d}\Big)_{(\mid C \mid)}\prod_{c \in C}(1-d)_{(\mid c \mid - 1)}.
\end{align*}
\end{theorem}

The proof of Theorem~\ref{theorem:main} follows the lines of \cite{miller2018elementary}'s derivation. We need the next two technical results, which we will prove in Section~\ref{section:proofs}. Let $C_z$ denote the partition $[n]$ induced by $z$ for any $z\in\N^n.$ Let ${k_n}$ be the number of parts in the partition. We define $m(z)= \text{ max }\{ z_{1},\ldots,z_{n} \}, \text{ and } g_j(z)=\#\{i:z_i\geq j\}.$

\begin{proposition}\label{prop:A}
  For any $z\in\N^n$, the marginal  distribution of the allocation variables vector $\z=(\z_1,\ldots,\z_n)$ is given by
$$\Pr(\z=z) = \frac{1}{(\alpha)_{(n)}}\prod_{c\in C_z}\frac{\Gamma(|c|+1-d)}{\Gamma(1-d)}\prod_{j=1}^{m(z)} \frac{\alpha+(j-1)d}{g_j(z)+\alpha+(j-1)d}.$$
\end{proposition}

\begin{lemma}\label{lemma:B}
    For any partition $C$ of $[n]$,
    $$\sum_{z\in\N^n} \I(C_z=C) \prod_{j=1}^{m(z)} \frac{\alpha+(j-1)d}{g_j(z)+\alpha+(j-1)d}=\frac{d^{|C|}}{\prod_{c \in C}(|c|-d)}\Big(\frac{\alpha}{d}\Big)_{(|C|)}.$$
\end{lemma}

\begin{proof}[Proof of Theorem \ref{theorem:main}]
\begin{equation*}
    \begin{split}
        \Pr(\C=C)&=\sum_{z \in \N^{n}}\Pr(\C=C | \z=z) \Pr(\z=z)\\
&\overset{\text{(a)}}{=}\sum_{z\in\N^n}\I(C_z = C)\frac{1}{(\alpha)_{(n)}}\prod_{c \in C_{z}} \frac{\Gamma (| c | + 1 - d)}{ \Gamma (1-d)} \prod_{j=1}^{m(z)} \frac{\alpha + (j-1)d}{g_{j}(z)+\alpha+(j-1)d}\\
&= \frac{1}{(\alpha)_{(n)}}\prod_{c \in C} \frac{\Gamma (| c | + 1 - d)}{ \Gamma (1-d)} \sum_{z\in\N^n}\I(C_z = C) \prod_{j=1}^{m(z)} \frac{\alpha + (j-1)d}{g_{j}(z)+\alpha+(j-1)d}\\
&\overset{\text{(b)}}{=} \frac{1}{(\alpha)_{(n)}}\prod_{c \in C} \frac{\Gamma (| c | + 1 - d)}{ \Gamma (1-d)} \frac{d^{|C|}}{\prod_{c \in C}(|c|-d)}\Big(\frac{\alpha}{d}\Big)_{(|C|)}\\
&\overset{\text{(c)}}{=} \frac{1}{(\alpha)_{(n)}} \prod_{c \in C}(1-d)_{(| c | - 1)} \prod_{c \in C} ( | c | - d )  \frac{d^{|C|}}{\prod_{c \in C}(|c|-d)}\Big(\frac{\alpha}{d}\Big)_{(|C|)}\\
&=\frac{d^{|C|}}{(\alpha)_{(n)}}  \Big( \frac{\alpha}{d}\Big)_{(|C|)} \prod_{c \in C} (1-d)_{(|c| - 1)},
    \end{split}
\end{equation*}
% \begin{align*}
% \end{align*}
where (a) is by Proposition~\ref{prop:A}, (b) is by Lemma \ref{lemma:B}, and (c) is since $\Gamma(|c|+1-d) =(|c |-d)\Gamma(|c-d|)$.
\end{proof}

\section{Proofs of the technical results}
\label{section:proofs}

\subsection{Additional lemmas}

We require the following additional lemmas.

\begin{lemma}\label{lemma:E}
For $a+c > 0, \text{ and } b + d>0$, if $\y \sim \mathrm{Beta}(a,b),$ then $\E [\y^{c}(1-\y)^{d}]=\frac{B(a+c, b+d)}{B(a,b)}$ where $B$ denotes the beta function.
\end{lemma}

\begin{proof}
\begin{align*}
\E [\y^{c}(1-\y)^{d}] &= \int_0^1 y^{c}(1-y)^{d} \frac{1}{B(a,b)}y^{a-1}(1-y)^{b-1} \mathrm{d}y\\
&=\frac{1}{B(a,b)}\int_0^1 y^{a+c-1}(1-y)^{b+d-1} \mathrm{d}y\\
&= \frac{B(a+c,b+d)}{B(a,b)}.
\end{align*}
\end{proof}

Let $S_{k_n}$ denote the set of ${k_n}!$ permutations of $[{k_n}]$. The following lemma is key for proving Lemma~\ref{lemma:B}.

\begin{lemma}\label{lemma:C}
For any $n_1,\ldots,n_{k_n}\in\N$,
$$\sum_{\sigma \in S_{{k_n}}} \prod_{i=1}^{{k_n}} \frac{1}{a_{i}(\sigma)-({k_n}-i+1)d} = \frac{1}{\prod_{i=1}^{{k_n}} (  n_{i}  - d) } $$
where $a_i(\sigma) = n_{\sigma_i} + n_{\sigma_{i +1}} +\cdots + n_{\sigma_{k_n}}$.
\end{lemma}

\begin{proof}
Consider the process of sampling without replacement ${k_n}$ times from an urn containing ${k_n}$ balls. The balls have sizes $n_1-d,\ldots,n_{k_n}-d,$ and the probability of drawing ball $i$ is proportional to its size $n_i-d.$ Thus for any permutation $\sigma\in S_{k_n}$ we have that

\begin{align*}
& p(\sigma_1) = \frac{n_{\sigma_{1}}-d}{n-td} =\frac{n_{\sigma_1}-d}{a_1(\sigma)-td},\\
& p(\sigma_2|\sigma_1) = \frac{n_{\sigma_2}-d}{n-n_{\sigma_{1}}-({k_n}-1)d} =\frac{n_{\sigma_2}-d}{a_2(\sigma)-({k_n}-1)d},\\
& p(\sigma_i|\sigma_1,\dots,\sigma_{i-1}) =\frac{n_{\sigma_i}-d}{n-n_{\sigma_{1}}-\cdots - n_{\sigma_{i-1}}-({k_n}-i+1)d} =\frac{n_{\sigma_i}-d}{a_i(\sigma)-({k_n}-i+1)d}.
\end{align*}
Therefore,
\begin{align}\label{equation:sigma}
p(\sigma) = p(\sigma_1) p(\sigma_2|\sigma_1)\cdots p(\sigma_{k_n}|\sigma_1,\ldots,\sigma_{{k_n} -1})
= \prod_{i=1}^{{k_n}}\frac{n_{\sigma_i} -d}{a_{i}(\sigma)-({k_n}-i+1)d}.
\end{align}
This way, we construct a distribution on $S_{k_n}.$ We know that $\sum_{\sigma\in S_{k_n}} p(\sigma) = 1.$ Applying this to Equation \eqref{equation:sigma} and dividing both sides by %$\prod_{i=1}^{{k_n}}(n_i - d)$
 $(n_{\sigma_1}-d)\cdots (n_{\sigma_{k_n}}-d) = (n_1-d)\cdots (n_{k_n}-d)$ gives the result.
\end{proof}

\begin{lemma}\label{lemma:D}

Let $b_i \in \N \text{ for } i \in \{1, \ldots, {k_n} \}$ and let $b_0 =0.$ We define $\bar{b}_{i} = b_0 + b_1 + \cdots + b_i.$
Then
\begin{equation*}
\prod_{i=1}^{{k_n}} \sum_{b_{i} \in \N} \frac{ (\frac{ \alpha }{ d } + \bar{b}_{i-1} )_{(b_{i})} }{ (\frac{a_{i}+\alpha}{d}+ \bar{b}_{i-1} )_{(b_{i})}} = \frac{(\frac{\alpha}{d})_{({k_n})}}{\prod_{i=1}^{{k_n}}(\frac{a_i}{d}-({k_n}+1-i))}.   
\end{equation*}
\end{lemma}

\begin{proof}
Let $A_j$ denote the intermediate sum $A_j = \prod_{i=j}^{{k_n}} \sum_{b_{i} \in \mathbb{N}} \frac{ (\frac{ \alpha }{ d } + \bar{b}_{i-1} )_{(b_{i})} }{ (\frac{a_{i}+\alpha}{d}+ \bar{b}_{i-1} )_{(b_{i})}}$. 
We show by induction decreasing from $j={k_n}$ to $j = 0$ that
\begin{align}
\label{ref9}
A_j =  \frac{(\frac{\alpha}{d} +  \bar{b}_{j-1})_{({k_n}-j+1)}}{\prod_{i=j}^{{k_n}}(\frac{a_{i}}{d}-({k_n}+1-i))}.
\end{align}

When $j={k_n}$ we have 
\begin{align*}
A_{k_n} = \sum_{b_{{k_n}} \in \mathbb{N}} \frac{ (\frac{ \alpha }{ d } + \bar{b}_{{k_n}-1} )_{(b_{{k_n}})} }{ (\frac{a_{{k_n}}+\alpha}{d}+ \bar{b}_{{k_n}-1} )_{(b_{{k_n}})}} = \sum_{b_{{k_n}} \in \mathbb{N}} \E [X^{b_{{k_n}}}]
\end{align*}
where $X \sim \text{Beta}(\frac{\alpha}{d}+ \bar{b}_{{k_n}-1}, \frac{a_{{k_n}}}{d})$. We have that
\begin{align*}
\sum_{b_{{k_n}} \in \mathbb{N}} \E [X^{b_{{k_n}}}] &=  \E \Big[ \sum_{b_{{k_n}} \in \mathbb{N}} X^{b_{{k_n}}}\Big]= \mathbb{E}\left[\frac{X}{1-X} \right] = \frac{\alpha + d \bar{b}_{{k_n}-1}}{a_{{k_n}}-d},
\end{align*}
due to Lemma \ref{lemma:E}, which proves the initialization for~\eqref{ref9}.

We now consider the case of an arbitrary $j$, greater than $0$ and less than ${k_n}$. By the induction hypothesis, we have that Equation~\eqref{ref9} holds for $j+1$, that is
\begin{align*}
A_{j+1} =   \frac{ ( \frac{\alpha}{d} +  \bar{b}_{j})_{({k_n}-j)}} {\prod_{i=j+1}^{{k_n}} (\frac{a_{i}}{d} - ({k_n}+1-i))}.
\end{align*}
Therefore,
\begin{align*}
A_j
&= \sum_{b_{j} \in \mathbb{N} } \frac{ (\frac{ \alpha }{ d } + \bar{b}_{j-1} )_{(b_{j})} }{ (\frac{a_{j}+\alpha}{d}+ \bar{b}_{j-1} )_{(b_{j})}} \prod_{i=j+1}^{{k_n}} \sum_{b_{i} \in \mathbb{N}} \frac{ (\frac{ \alpha }{ d } + \bar{b}_{i-1} )_{(b_{i})} }{ (\frac{a_{i}+\alpha}{d}+ \bar{b}_{i-1} )_{(b_{i})}} \\
&= \sum_{b_{j} \in \mathbb{N} } \frac{(\frac{\alpha}{d}+\bar{b}_{j-1})_{(b_{j})}}{(\frac{a_{j}}{d}+\frac{\alpha}{d}+ \bar{b}_{j-1})_{(b_{j})}}  \frac{ (\frac{\alpha}{d} +  \bar{b}_{j})_{({k_n}-j)}} {\prod_{i=j+1}^{{k_n}} (\frac{a_i}{d} - ({k_n}+1-i))}
\end{align*}

Rearranging the rising factorials in the numerator, we can write
\begin{align*}
    \Big( \frac{\alpha}{d}+\bar{b}_{j-1}\Big)_{(b_{j})}\Big(\frac{\alpha}{d} +  \bar{b}_{j}\Big)_{({k_n}-j)} 
    &= \Big(\frac{\alpha}{d}+\bar{b}_{j-1}\Big)_{(b_{j})}\Big(\frac{\alpha}{d} +  \bar{b}_{j-1}+b_j \Big)_{({k_n}-j)}\\
    &= \Big(\frac{\alpha}{d}+\bar{b}_{j-1}\Big)_{(b_{j}+{k_n}-j)}\\
    &= \Big(\frac{\alpha}{d}+\bar{b}_{j-1}\Big)_{({k_n}-j)}\Big(\frac{\alpha}{d} +  \bar{b}_{j-1}+{k_n}-j\Big)_{(b_j)} 
\end{align*}
and thus factorize the terms independent of $b_j$ in order to obtain
\begin{align*}
A_j
&=  \frac{ (\frac{\alpha}{d}+\bar{b}_{j-1})_{({k_n}-j)}} {\prod_{i=j+1}^{{k_n}} (\frac{a_i}{d} - ({k_n}+1-i))}
\sum_{b_{j} \in \mathbb{N}}
\frac{(\frac{\alpha}{d}+\bar{b}_{j-1}+{k_n}-j)_{(b_{j})}}{(\frac{a_{j}}{d}+\frac{\alpha}{d}+\bar{b}_{j-1})_{(b_{j})}}.
\end{align*}
The sum above can be rewritten, using $X \sim \text{Beta}(\frac{\alpha}{d}+\bar{b}_{j-1}+({k_n}-j),\frac{a_{j}}{d}-({k_n}-j))$, as
\begin{align*}
\sum_{b_{j} \in \mathbb{N}} \mathbb{E} [X^{b_{j}}] = \mathbb{E} \left[\frac{X}{1-X}\right]
=\frac{ \frac{\alpha}{d} + \bar{b}_{j-1} + ({k_n}-j) }{ \frac{a_{j}}{d} - ({k_n}+1-j) }.
\end{align*}
Putting this all together,
\begin{align*}
A_j
&= \frac{ (\frac{\alpha}{d} + \bar{b}_{j-1} + ({k_n}-j)) }{ \frac{a_{j}}{d} - ({k_n}+1-j) }
\frac{ (\frac{\alpha}{d}+\bar{b}_{j-1})_{({k_n}-j)}} {\prod_{i=j+1}^{{k_n}} (\frac{a_i}{d} - ({k_n}+1-i))}\\
&= \frac{(\frac{\alpha}{d} +  \bar{b}_{j-1})_{({k_n}-j+1)}}{\prod_{i=j}^{{k_n}}(\frac{a_i}{d}-({k_n}+1-i))}
\end{align*}
which proves the desired result for $j$.
By induction, this result is true for all $j \in \{ 1,\dots,{k_n} \}.$ Letting $j=1$ gives the result stated in the lemma, since $\bar{b}_{0}=b_0=0$.
\end{proof}

\subsection{Proof of Proposition \ref{prop:A} and Lemma \ref{lemma:B}}

\begin{proof}[Proof of Proposition \ref{prop:A}]
For simplicity, we fix the allocation variable vector to a value $z$ and denote $m(z)$ by $m$ and $g_{j}(z)$ by $g_{j}.$ We have
$$\Pr(\boldsymbol{z} = z | \pi_1,\ldots,\pi_m) =\prod_{i = 1}^n\pi_{z_i} =\prod_{j = 1}^m \pi_j^{e_j} $$
where $e_j =\#\{i:z_i = j\}$.
Thus,
$$\Pr(\boldsymbol{z} = z | v_1,\ldots,v_m) =\prod_{j = 1}^m\Big(v_j\textstyle\prod_{i = 1}^{j -1} (1 - v_i)\Big)^{e_j}
=\displaystyle\prod_{j = 1}^m v_j^{e_j} (1 - v_j)^{f_j} $$
where $f_j =\#\{i:z_i>j\}$. Therefore,

\begin{align*}
\Pr(\z = z) & =\int \Pr(\z = z | v_1,\ldots,v_m) p(v_1,\ldots,v_m) \ddr v_1\cdots \ddr v_m\\
& =\int \Big(\prod_{j = 1}^m v_j^{e_j} (1 - v_j)^{f_j}\Big) p_{1}(v_1)\cdots p_{m}(v_m) \ddr v_1\cdots \ddr v_m\\
& =\prod_{j = 1}^m \int v_j^{e_j} (1 - v_j)^{f_j} p_{j}(v_j) \ddr v_j\\
& \overset{\text{(a)}}{=}\prod_{j = 1}^m \frac{B(e_{j}+1-d,f_{j}+\alpha+jd)}{B(1-d,\alpha+jd)}\\
&= \prod_{j=1}^{m}\frac{\Gamma(e_{j}+1-d)\Gamma(f_{j}+\alpha+jd) \Gamma(\alpha + (j-1)d)+1}{\Gamma(e_{j}+f_{j}+\alpha+(j-1)d)+1)\Gamma(1-d)\Gamma(\alpha + jd)}\\
& \overset{\text{(b)}}{=} \prod_{j=1}^{m} \frac{\Gamma(e_{j}+1-d)}{\Gamma(1-d)} \prod_{j=1}^{m} \frac{\Gamma(g_{j+1}+\alpha+jd)}{\Gamma(g_{j}+\alpha+(j-1)d+1)} \prod_{j=1}^{m} \frac{\Gamma(\alpha+(j-1)d+1)}{\Gamma(\alpha+jd)}\\
& \overset{\text{(c)}}{=} \prod_{j=1}^{m} \frac{\Gamma(e_{j}+1-d)}{\Gamma(1-d)} \prod_{j=1}^{m} \frac{\alpha + (j-1)d}{g_{j}+\alpha+(j-1)d} \frac{\Gamma (g_{m+1}+\alpha+md) \Gamma(\alpha)}{\Gamma(g_1+\alpha) \Gamma(\alpha + md)}\\
& \overset{\text{(d)}}{=}\frac{\Gamma (\alpha)}{\Gamma ( n + \alpha)} \prod_{c \in C_{z}} \frac{\Gamma (| c | + 1 - d)}{ \Gamma (1-d)} \prod_{j=1}^{m} \frac{\alpha + (j-1)d}{g_{j}+\alpha+(j-1)d}
\end{align*}
where step (a) follows from Lemma \ref{lemma:E}, step (b) since $f_j = g_{j +1}$ and $g_j = e_j + f_j$, step (c) since $\Gamma(x +1) = x\Gamma(x)$, and step (d) since $g_{1}=n$ and $g_{m+1}=0.$
\end{proof}

\begin{proof}[Proof of Lemma \ref{lemma:B}]
As before, we denote the parts of $C$ by $c_1,\dots,c_{k_n},$ and we let ${k_n}=|C|.$ We denote the distinct values taken on by $z_1,\dots,z_n$ by $j_1 < \dots < j_{k_n}.$ We define $j_{0}=b_{0}=0,$ $b_{i} = j_{i}-j_{i-1},$ and $\bar{b}_{i}=b_{0}+\cdots+b_{i}$ for $i \in \{ 1,\ldots,{k_n} \}$. We use the notation $a_{i}(\sigma)=n_{\sigma_{i}}+\cdots+n_{\sigma_{{k_n}}},$ where $\sigma$ is the permutation of $[{k_n}]$ such that $c_{\sigma_{i}}=\{ \ell : z_\ell = j_i \}.$ Then for any $z \in \mathbb{N}^{n}$ such that $C_{z}=C,$

\begin{align*}
\prod_{j=1}^{m(z)} \frac{\alpha+(j-1)d}{g_{j}(z)+\alpha+(j-1)d} &= \prod_{j=1}^{m(z)} \frac{\frac{\alpha}{d}+j-1}{\frac{g_{j}(z)+\alpha}{d}+j-1}\\
=\prod_{i=1}^{{k_n}} \prod_{j=\bar{b}_{i-1}+1}^{\bar{b}_{i}} \frac{ \frac{\alpha}{d}+j-1 }{\frac{g_{j}(z)+\alpha}{d}+j-1} &=\prod_{i=1}^{{k_n}}\frac{(\frac{\alpha}{d}+\bar{b}_{i-1})_{(b_{i})}}{(\frac{\alpha + a_{i}(\sigma)}{d}+\bar{b}_{i-1})_{(b_{i})}},
\end{align*}
because $g_{j}(z)=a_{i}(\sigma)$ for $\bar{b}_{i-1} < j \leq \bar{b}_{i}.$ It follows from the definition of $b=(b_{1},...,b_{{k_n}})$ and $\sigma$ that there is a one-to-one correspondence between $ \{ z \in \mathbb{N}^{n}:C_{z} = C \}$ and $\{ (\sigma,b): \sigma \in S_{{k_n}}, b \in \mathbb{N}^{{k_n}} \}.$ Therefore,
\begin{align*}
\sum_{z \in \mathbb{N}^{n}} \I(C_{z}=C) &\prod_{j=1}^{m(z)} \frac{\alpha + (j-1)d}{g_{j}(z)+ \alpha + (j-1)d} = \sum_{\sigma \in S_{{k_n}}} \sum_{b \in \mathbb{N}^{{k_n}}} \prod_{i=1}^{{k_n}} \frac{ (\frac{ \alpha }{ d } + \bar{b}_{i-1} )_{(b_{i})} }{ (\frac{a_{i}(\sigma)+\alpha}{d}+ \bar{b}_{i-1} )_{(b_{i})}}\\
 &= \sum_{\sigma \in S_{{k_n}}} \prod_{i=1}^{{k_n}} \sum_{b_{i} \in \mathbb{N}}  \frac{ (\frac{ \alpha }{ d } + \bar{b}_{i-1} )_{(b_{i})} }{ (\frac{a_{i}(\sigma)+\alpha}{d}+ \bar{b}_{i-1} )_{(b_{i})}}\\
 & \overset{\text{(a)}}{=} \sum_{\sigma \in S_{{k_n}}} \prod_{i=1}^{{k_n}}  \frac{(\frac{\alpha}{d})_{({k_n})}}{\frac{a_{i}(\sigma)}{d}-({k_n}-i+1)}\\
 &= d^{{k_n}}\Big(\frac{\alpha}{d}\Big)_{({k_n})}\sum_{\sigma \in S_{{k_n}}} \prod_{i=1}^{{k_n}}\frac{1}{a_{i}(\sigma)-({k_n}-i+1)d}\\
 & \overset{\text{(b)}}{=} \frac{d^{{k_n}} }{\prod_{c \in C}(|c|-d)}\Big(\frac{\alpha}{d}\Big)_{({k_n})},
\end{align*}
where step (a) follows from Lemma \ref{lemma:D} and step (b) follows from Lemma \ref{lemma:C}.
\end{proof}

\section*{Acknowledgement}
The authors would like to thank \href{http://www.bernardonipoti.com/}{Bernardo Nipoti} for fruitful discussions that initiated this work.

\bibliographystyle{apalike}
\bibliography{biblio}

\end{document}